\documentclass{amsart}
\usepackage{amsfonts, amsbsy, amsmath, amssymb}
\newtheorem{thm}{Theorem}[section]
\newtheorem{lem}[thm]{Lemma}

\newtheorem{exmp}[thm]{Example}

\numberwithin{equation}{section}

\theoremstyle{definition}

\input prepictex.tex
\input pictex.tex
\input postpictex.tex

\begin{document}

\title{Galkin Quandles, Pointed Abelian Groups, and Sequence $A000712$}

\author{W. Edwin Clark}
\address{Department of Mathematics  and Statistics,
University of South Florida, Tampa, FL 33620}
\email{wclark@mail.usf.edu}

\author{Xiang-dong Hou}
\address{Department of Mathematics  and Statistics,
University of South Florida, Tampa, FL 33620}
\email{xhou@usf.edu}


 
\begin{abstract}
For each pointed abelian group $(A,c)$, there is an associated {\em Galkin quandle} $G(A,c)$ which is an algebraic structure defined on $\Bbb Z_3\times A$ that can be used to construct knot invariants. It is known that two finite Galkin quandles are isomorphic if and only if their associated pointed abelian groups are isomorphic. In this paper we classify all finite pointed abelian groups. We show that the number of nonisomorphic pointed abelian groups of order $q^n$ ($q$ prime) is $\sum_{0\le m\le n}p(m)p(n-m)$, where $p(m)$ is the number of partitions of integer $m$. 
\end{abstract}

\maketitle

\section{Introduction}

The purpose of this paper is to demonstrate some nice connections between the three objects in the title: Galkin quandle, pointed abelian group, and sequence $A000712$ 
(number of partitions of $n$ into parts of $2$ kinds). First, let us describe the three objects briefly.

\subsection*{Galkin quandles}\

A {\em quandle} is a set $X$ equipped with an operation $*$ satisfying the following conditions. 
\begin{itemize}
\item[(i)] For each $x\in X$,
$x*x=x$.

\item[(ii)]
For each $y\in X$, the mapping $x\mapsto x*y$ is a permutation of $X$.

\item[(iii)]  For all $x,y,z\in X$, 
$(x*y)*z=(x*z)*(y*z)$.
\end{itemize} 

A {\em coloring} of a knot diagram (directed) by a quandle $(X,*)$ is a labeling of the arcs of the diagram by the elements of $X$ such that at each crossing the rule depicted in Figure~\ref{F1} is observed. The number of colorings of a knot $K$ by a quandle $X$, denoted by $N_X(K)$, is a knot invariant that can be used to distinguish nonequivalent knots \cite{CKS, CEHSY}

\bigskip

\begin{figure}[h]
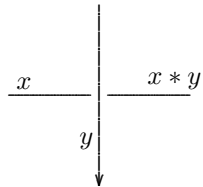
\label{F1}
\[
\beginpicture
\setcoordinatesystem units <4mm,4mm> point at 0 0
\arrow <4pt> [0.3, 0.67] from 0 3 to 0 -3
\setlinear 
\plot -3 0  -0.3 0 /
\plot 0.3 0  3 0 /
\put {$x$} [b] at -2.5 0.2
\put {$x*y$} [b] at 2.5 0.2 
\put {$y$} [r] at -0.2 -1.5
\endpicture
\] 
\medskip
\caption{Colors at a crossing}
\end{figure}


Define $\mu,\tau:\Bbb Z_3\to\Bbb Z$ by
\[
\mu(x)=\begin{cases} 
2&\text{if}\ x=0,\cr
-1&\text{if}\ x\ne 0,
\end{cases}
\qquad
\tau(x)=\begin{cases} 
1&\text{if}\ x=2,\cr
0&\text{if}\ x\ne 2.
\end{cases}
\]
Let $A$ be an abelian group and $c\in A$. For $(x,a),(y,b)\in \Bbb Z_3\times A$, define
\[
(x,a)*(y,b)=(-x-y,\ -a+\mu(x-y)b+\tau(x-y)c).
\]
The {\em Galkin quandle}, denoted by $G(A,c)$, is the structure $(\Bbb Z_3\times A,*)$. This construction was first given by Galkin in \cite{Gal88} for $A=\Bbb Z_p$ and was recently generalized to any abelian group $A$ in \cite{CEHSY}. For more properties of Galkin quandles, see \cite{CEHSY}.

\subsection*{Pointed abelian groups}\

A {\em pointed abelian group} is a pair $(A,c)$ where $A$ is an abelian group and $c\in A$. A morphism from a pointed abelian group $(A,c)$ to another pointed abelian group $(A',c')$ is a homomorphism $f:A\to A'$ such that $f(c)=c'$. The category of pointed abelian groups is denoted by ${\bf Ab}_0$.

\subsection*{Sequence $A000712$}\

In the On-Line Encyclopedia of Integer Sequences \cite{OEIS}, $A000712$ is the sequence $\{a(n)\}_{n=0}^\infty$, where $a(n)$ is number of partitions of $n$ into parts of $2$ kinds. One has 
\[
\sum_{n=0}^\infty a(n)x^n=\prod_{m=1}^\infty\frac 1{(1-x^m)^2}
\]
and
\[
a(n)=\sum_{0\le m\le n}p(m)p(n-m),
\]
where $p(m)$ is the number of partitions of $m$. For many other interpretations of $a(n)$, see \cite{OEIS}. 

Let ${\bf Q}$ denote the category of quandles. For each pointed abelian group $(A,c)$, define $\mathcal F(A,c)=G(A,c)$, and for each morphism $f:(A,c)\to (A',c')$ of pointed abelian groups, let $\mathcal Ff:G(A,c)\to G(A',c')$ be the quandle homomorphism defined by
\[
(\mathcal Ff)(x,a)=(x,f(a)),\qquad (x,a)\in\Bbb Z_3\times A.
\]
Then $\mathcal F$ is a functor from ${\bf Ab}_0$ to ${\bf Q}$ \cite{CEHSY}. In particular, if $(A,c)$ and $(A',c')$ are isomorphic pointed abelian groups, then $G(A,c)$ and $G(A',c')$ are isomorphic quandles. For finite Galkin quandles, the converse is also true: If $G(A,c)$ and $G(A',c')$ are isomorphic quandles, then $(A,c)$ and $(A',c')$ are isomorphic pointed abelian groups \cite{CEHSY}. Therefore, classification of finite Galkin quandles is the same as classification of finite pointed abelian groups.

Let $(A,c)$ be a finite pointed abelian group. We use $q$ to denote a prime to avoid confusion with the partition number $p(n)$.
For each prime $q$, denote the $q$-part of $A$ by $A_q$. ($A_q=\{a\in A:q^na=0$ for some $n\ge 0\}$.) Then $A=\bigoplus_qA_q$. Write $c=\sum_q c_q$, $c_q\in A_q$. Then
\[
(A,c)=\bigoplus_q(A_q,c_q),
\]
where the meaning of the direct sum of pointed abelian groups is self explaining. If $(A',c')$ is a another finite pointed abelian group, then $(A,c)\cong (A',c')$ if and only if $(A_q,c_q)\cong (A_q',c_q')$ for all primes $q$. Therefore, to classify all finite pointed abelian groups, it suffices to classify all finite pointed abelian $q$-groups.

Let $A$ be a finite abelian $q$-group and let $\text{Aut}(A)$ act on $A$ naturally. For $c,c'\in A$, the pointed abelian groups $(A,c)$ and $(A,c')$ are isomorphic if and only if $c$ and $c'$ are in the same $\text{Aut}(A)$-orbit of $A$. The automorphism group $\text{Aut}(A)$ is well known \cite{Hil-Rhe07, Ran07}. In section 2, we describe the orbit representatives of $A$ under the action of $\text{Aut}(A)$. Thus a classification of finite pointed abelian $q$-groups is obtained. This classification allows us to compute the number $N(n)$ of nonisomorphic pointed abelian groups of order $q^n$. ($N(n)$ is independent of $q$.)   
The initial formula for $N(n)$ resulting from the classification is rather complicated and does not suggest any connection to any well known sequence. However, a search through the On-Line Encyclopedia of Integer Sequences (OEIS) shows that the numerical values of $N(n)$ that we have computed match the sequence $A000712$. In section 3, we confirm that $N(n)$ is indeed the sequence $A000712$. The proof is rather tricky; the key step is a formula for the partition number $p(n)$ based on a slight variation of the {\em Frobenius symbol} of a partition.


\section{Classification of Finite Pointed Abelian $q$-Groups}

Let $q$ be a prime and 
\begin{equation}\label{2.1}
A=A_1\oplus\cdots\oplus A_k,
\end{equation}
where $A_i=\Bbb Z_{q^{e_i}}^{n_i}$, $1\le e_1<\cdots<e_k$ and $n_i>0$, $1\le i\le k$. Let $\pi_i:A\to A_i$ be the projection and $\iota_j:A_j\to A$ be the inclusion. Then 
\[
\text{End}_\Bbb Z(A)=\bigoplus_{i,j}\iota_j\text{Hom}_\Bbb Z(A_i,A_j)\pi_i.
\]
The mapping $\iota_j(\ )\pi_i:\text{Hom}_\Bbb Z(A_i,A_j)\to \iota_j\text{Hom}_\Bbb Z(A_i,A_j)\pi_i$ is an isomorphism. We will identify $\text{Hom}_\Bbb Z(A_i,A_j)$ with
$\iota_j\text{Hom}_\Bbb Z(A_i,A_j)\pi_i$ through this isomorphism. Thus we can write
\[
\text{End}_\Bbb Z(A)=\bigoplus_{i,j}\text{Hom}_\Bbb Z(A_i,A_j).
\]
For $\sigma=\sum_{i,j}\sigma_{ij}\in \text{End}_\Bbb Z(A)$, where $\sigma_{ij}\in \text{Hom}_\Bbb Z(A_i,A_j)$, it is well known that $\sigma\in\text{Aut}(A)$ if and only if $\sigma_{ii}\in\text{Aut}(A_i)$ for all $1\le i\le k$ \cite{Hil-Rhe07, Ran07}.

\begin{thm}\label{T2.1}
Let $A$ be a finite abelian $q$-group written in the form \eqref{2.1}. For each $1\le i\le k$, choose $\epsilon_i\in A_i\setminus qA_i$. Let $\mathcal I(e_1,\dots,e_k)$ be the set whose elements are sequence of integer pairs $(i_1,f_1),\dots,(i_l,f_l)$ satisfying 
\begin{itemize}
  \item [(i)] $l\ge 0$, (the sequence is empty when $l=0$,)
  \item [(ii)] $1\le i_1<\cdots<i_l\le k$,
  \item [(iii)] $0\le f_s\le e_{i_s}-1,\ 1\le s\le l$,
  \item [(iv)] $0<f_{s+1}-f_s<e_{i_{s+1}}-e_{i_s},\ 1\le s\le l-1$.
\end{itemize}  
Then as $(i_1,f_1),\dots,(i_l,f_l)$ runs through $\mathcal I(e_1,\dots,e_k)$, 
\begin{equation}\label{2.2}
\sum_{s=1}^lq^{f_s}\epsilon_{i_s}
\end{equation}
gives a complete list of orbit representatives of $A$ under the action of $\text{\rm Aut}(A)$.
\end{thm}

\begin{proof}
$1^\circ$
We first show that for each $x\in A$, there exists $\sigma\in\text{Aut}(A)$ such that $\sigma(x)$ is of the form \eqref{2.2}.

Clearly, the $\text{Aut}(A_i)$-orbits of $A_i$ are represented by $0$ and $q^t\epsilon_i$, $0\le t\le e_i-1$. Thus there exists $\sigma\in\text{Aut}(A)$ such that
\[
\sigma(x)=\sum_{1\le s\le l}q^{f_s}\epsilon_{i_s},
\]
where $1\le i_1<\cdots<i_l\le k$ and $0\le f_s\le e_{i_s}-1$, $1\le s\le l$. We further assume that $\sigma$ is so chosen that $l$ is {\em minimum}. We claim that 
\begin{equation}\label{2.3}
0<f_{s+1}-f_s<e_{i_{s+1}}-e_{i_s},\qquad 1\le s\le l-1.
\end{equation}

If, to the contrary of \eqref{2.3}, $f_{m+1}-f_m\le 0$ for some $1\le m\le l-1$, then there exists $\alpha\in\text{Hom}_\Bbb Z(A_{i_{m+1}},A_{i_m})$ such that $\alpha(q^{f_{m+1}}\epsilon_{i_{m+1}})
=q^{f_m}\epsilon_{i_m}$. Let $\beta=\text{id}_A-\alpha\in\text{Aut}(A)$. Then 
\[
\beta\sigma(x)=\beta\Bigl(\sum_{1\le s\le l}q^{f_s}\epsilon_{i_s}\Bigr)=\sum_{\substack{1\le s\le l\cr s\ne m}}q^{f_s}\epsilon_{i_s},
\]
which is a contradiction to the minimality of $l$.

If, to the contrary of \eqref{2.3}, $f_{m+1}-f_m\ge e_{i_{m+1}}-e_{i_m}$ for some $1\le m\le l-1$, then there exists $\alpha'\in\text{Hom}_\Bbb Z(A_{i_m},A_{i_{m+1}})$ such that $\alpha'(q^{f_m}\epsilon_{i_m})=q^{f_{m+1}}\epsilon_{i_{m+1}}$. Let $\beta'=\text{id}_A-\alpha'\in\text{Aut}(A)$. Then
\[
\beta'\sigma(x)=\sum_{\substack{1\le s\le l\cr s\ne m+1}}q^{f_s}\epsilon_{i_s},
\]
which is again a contradiction.

So \eqref{2.3} is proved. Now $\sigma(x)$ is of the form \eqref{2.2} and the sequence $(i_1,f_1),\dots,(i_l,f_l)$ satisfies conditions (i) -- (iv).

$2^\circ$
Assume that $(i_1,f_1),\dots,(i_l,f_l)$ and $(j_1,g_1),\dots,(j_m,g_m)$ are two different sequences in $\mathcal I(e_1,\dots,e_k)$. We show that $\sum_{1\le s\le l}q^{f_s}\epsilon_{i_s}$ and $\sum_{1\le s\le m}q^{g_s}\epsilon_{j_s}$ do not belong to the same $\text{Aut}(A)$-orbit.

Without loss of generality, we may assume that there exists $1\le u\le l$ such that either $i_u\notin\{j_1,\dots,j_m\}$ or $i_u=j_v$ for some $1\le v\le m$ but $f_u<g_v$.
For each $\sigma\in\text{Aut}(A)$ written in the form $\sigma=\sum_{i,j}\sigma_{ij}$, where $\sigma_{ij}\in\text{Hom}_\Bbb Z(A_i,A_j)$, $\sigma_{ii}\in\text{Aut}(A_i)$, the $A_{i_u}$-component of $\sigma(\sum_{1\le s\le l}q^{f_s}\epsilon_{i_s})$ is 
\[
\sum_{1\le i\le k}\sum_{1\le s\le l}\sigma_{i,i_u}(q^{f_s}\epsilon_{i_s}).
\]
Because of condition (iv), we have
\[
\sigma_{i,i_u}(q^{f_s}\epsilon_{i_s})\in 
\begin{cases}
q^{f_u}A_{i_u}\setminus q^{f_u+1}A_{i_u}&\text{if}\ s=u,\ i=i_u,\cr
q^{f_u+1}A_{i_u}&\text{otherwise}.
\end{cases}
\]
So the $A_{i_u}$-component of $\sigma(\sum_{1\le s\le l}q^{f_s}\epsilon_{i_s})$ belongs to $q^{f_u}A_{i_u}\setminus q^{f_u+1}A_{i_u}$. On the other hand, the $A_{i_u}$-component of $\sum_{1\le s\le m}q^{g_s}\epsilon_{j_s}$ belongs to $q^{f_u+1}A_{i_u}$. Thus
\[
\sigma\Bigl(\sum_{1\le s\le l}q^{f_s}\epsilon_{i_s}\Bigr)\ne \sum_{1\le s\le m}q^{g_s}\epsilon_{j_s}.
\]
\end{proof} 


\section{Number of Nonisomorphic Finite Pointed Abelian $q$-Groups}

It follows from Theorem~\ref{T2.1} that the number of nonisomorphic pointed abelian groups with the underlying group $\Bbb Z_{q^{e_1}}^{n_1}\oplus\cdots\oplus\Bbb Z_{q^{e_k}}^{n_k}$ ($1\le e_1<\cdots<e_k,\ n_i>0$) is $|\mathcal I(e_1,\dots, e_k)|$. Consequently, the number of nonisomorphic pointed abelian groups of order $q^n$, denoted by $N(n)$, is given by
\begin{equation}\label{3.1}
N(n)=\sum_{\substack{1\le e_1<\cdots<e_k\cr n_1,\dots,n_k>0\cr n_1e_1+\cdots+n_ke_k=n}}|\mathcal I(e_1,\dots,e_k)|.
\end{equation}
The values of $N(0), N(1), \dots$ are
\[
1,\ 2,\ 5,\ 10,\ 36,\ 65,\ 110,\ 185,\ 300,\ 481, \dots
\]
which point to the sequence $A000712$ as a possible match. However, in the form \eqref{3.1}, it is not clear that $N(n)$ is the sequence $A000712$.

\begin{thm}\label{T3.1}
$N(n)$ is the sequence $A000712$. Namely,
\begin{equation}\label{3.2}
N(n)=\sum_{0\le m\le n}p(m)p(n-m),\qquad n\ge 0.
\end{equation}
\end{thm}

The key step in the proof of Theorem~\ref{T3.1} is the following lemma.

\begin{lem}\label{T3.2}
For integers $k,l\ge 0$ define
\begin{gather*}
A(k,l)=\{(x_1,\dots,x_k)\in\Bbb Z^k:1\le x_1<\cdots<x_k,\ x_1+\cdots+x_k=l\},\\
B(k,l)=\{(x_1,\dots,x_k)\in\Bbb Z^k:0\le x_1<\cdots<x_k,\ x_1+\cdots+x_k=l\}.
\end{gather*}
Then 
\begin{equation}\label{3.3}
p(n)=\sum_{0\le k\le l\le n}|A(k,l)||B(k,n-l)|=\sum_{0\le k\le l\le n}|B(k,l-k)||B(k,n-l)|.
\end{equation}
\end{lem}

\begin{proof}
Since $|A(k,l)|=|B(k,l-k)|$, we only have to prove the first equal sign in \eqref{3.3}.

In fact, there is a bijection between the set of partitions of $n$ and 
\[
\bigcup_{0\le k\le l\le n}A(k,l)\times B(k,n-l).
\]
Given $(x_1,\dots,x_k)\in A(k,l)$ and $(y_1,\dots,y_k)\in B(k,n-l)$, we can build a Ferrers diagram herringbone style that corresponds to the a partition of $n$; see Figure~\ref{F2}. 

\begin{figure}[h]
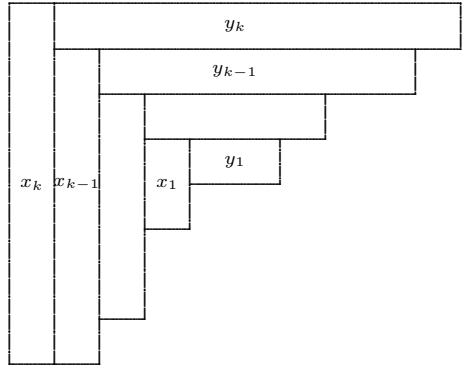

\[
\beginpicture
\setcoordinatesystem units <6mm,6mm> point at 0 0
\setlinear 
\plot 0 0  0 8  1 8  1 0  0 0 /
\plot 1 8  10 8  10 7  1 7 /
\plot 1 0  2 0  2 7 /
\plot 2 6  9 6  9 7 /
\plot 2 1  3 1  3 6 /
\plot 3 5  7 5  7 6 /
\plot 3 3  4 3  4 5 /
\plot 4 4  6 4  6 5 /
\put {$\scriptstyle x_k$} at 0.5 4
\put {$\scriptstyle x_{k-1}$} at 1.5 4 
\put {$\scriptstyle x_1$} at 3.5 4
\put {$\scriptstyle y_k$} at 5 7.5
\put {$\scriptstyle y_{k-1}$} at 5 6.5 
\put {$\scriptstyle y_1$} at 5 4.5
\endpicture
\] 
\vskip 2cm
\caption{The Ferrers diagram corresponding to $(x_1,\dots,x_k)$, $(y_1,\dots,y_k)$}\label{F2}
\end{figure}


Conversely, given a Ferrers diagram, we can retrieve two sequences $(x_1,\dots,x_k)\in A(k,l)$ and $(y_1,\dots,y_k)\in B(k,n-l)$ in the order of $x_k,y_k,x_{k-1},y_{k-1},\dots,x_1,y_1$ as depicted in Figure~\ref{F2}. (Note that it is necessary that $y_1$ be allowed to be $0$.) Therefore we have the desired bijection.
\end{proof}

\noindent{\bf Remark.}
The herringbone construction in Figure~\ref{F2} is a slight variation of the Frobenius symbol of a partition. The array
\[
\left(\begin{matrix} 
y_k & y_{k-1}&\cdots& y_1\cr
x_k-1& x_{k-1}-1&\cdots&x_1-1\end{matrix}\right)
\]
is the Frobenius symbol of the partition in Figure~\ref{F2} \cite{And84}. Thus Lemma~\ref{T3.2} is essentially counting of Frobenius symbols of partitions. For a comprehensive treatise of Frobenius symbols and generalized Frobenius partitions, see \cite{And84}.

\begin{proof}[Proof of Theorem~\ref{T3.1}]
Each isomorphic class of pointed abelian groups of order $q^n$ is uniquely determined by the following data:
\[
\begin{split}
&1\le e_1<\cdots<e_k,\ n_1,\dots,n_k>0\ \text{such that}\ n_1e_1+\cdots+n_ke_k=n;\cr
&\bigl((f_1,i_1),\dots,(f_l,i_l)\bigr)\in\mathcal I(e_1,\dots,e_k).
\end{split}
\]
Each set of such data can be obtained exactly once through the following steps;

\medskip

{\bf Step 1.} Choose $(x_1,\dots,x_l)\in A(l,u)$, $(y_1,\dots,y_l)\in B(l,m-u)$, where $0\le l\le u\le m\le n$.

\medskip

{\bf Step 2.} Choose a partition $\lambda$ of $n-m$. The union of $\lambda$ and $x_1+y_1,\dots,x_l+y_l$ is a partition $\mu$ of $n$. Write
\[
\mu=(\underbrace{e_1,\dots,e_1}_{n_1},\dots, \underbrace{e_k,\dots,e_k}_{n_k}),
\]
where $1\le e_1<\cdots<e_k$, $n_1,\dots,n_k>0$.

\medskip

{\bf Step 3.} Let $f_s=x_s$, $1\le s\le l$ and let $i_s$ be defined by $x_s+y_s=e_{i_s}$.

To observe how these steps are actually carried out, see Example~\ref{E3.3}.

For each $0\le m\le n$, the number of choices in step 1 is 
\[
\sum_{0\le l\le u\le m}|A(l,u)||B(l,m-u)|=p(m)\qquad\text{(by Theorem~\ref{T3.2})}.
\]
The number of choices in step 2 is $p(n-m)$ and the number of choices in step 3 is 1. Thus we have
\[
N(n)=\sum_{0\le m\le n}p(m)p(n-m).
\]
\end{proof} 

\begin{exmp}\label{E3.3}
\rm
Assume $n=20$.

{\em Example of step 1.} Choose $(1,3,5)\in A(3,9)$ and $(1,2,4)\in B(3,7)$. Note that $m=16$.

{\em Example of step 2.} $n-m=4$. Choose $\lambda=(1,1,2)\vdash 4$. Then
\[
\begin{split}
\mu&\;=(1,1,2,2,5,9),\cr
(e_1,e_2,e_3,e_4)&\;=(1,2,5,9),\cr
(n_1,n_2,n_3,n_4)&\;=(2,2,1,1).
\end{split}
\]

{\em Step 3.} We have $(f_1,f_2,f_3)=(1,3,5)$. Since $x_1+y_1=2=e_2$, we have $i_1=2$. In the same way, $(i_1,i_2,i_3)=(2,3,4)$.
\end{exmp}


\section*{Acknowledgment}

We thank Professor George E. Andrews for pointing out the connection of Lemma~\ref{T3.2} with the Frobenius symbol of a partition.


\end{document}